\documentclass[10pt]{amsart}
\usepackage{cmap,amsmath,amsthm,amssymb,amscd} 
\usepackage[utf8]{inputenc}
\usepackage[english]{babel}
\usepackage{mathrsfs} 
\usepackage{placeins}
\usepackage[T1]{fontenc}
\usepackage{csquotes}
\usepackage[style=gost-numeric,giveninits=true, doi=false,isbn=false,url=false,eprint=false, defernumbers=true, maxbibnames=99]{biblatex}
\usepackage{xcolor}
\usepackage{enumerate}
\usepackage{verbatim}
\usepackage{amsbsy}
\usepackage{graphics}
\usepackage[caption = false]{subfig}
\usepackage[final]{graphicx}
\usepackage{grffile}
\usepackage{mathtools}
\usepackage{tikz-cd}
\usetikzlibrary{decorations.pathmorphing}

\newtheorem*{theorem**}{Theorem\theoremnum}
\newenvironment{theorem*}[1][]{%
  \edef\theoremnum{\if\relax\detokenize{#1}\relax\else~#1\fi}% Store theorem number
  \begin{theorem**}
}{%
  \end{theorem**}
}

\newtheorem{theorem}{Theorem}[section]
\newtheorem{proposition}[theorem]{Proposition}
\newtheorem{corollary}[theorem]{Corollary}
\newtheorem{lemma}[theorem]{Lemma}

\newtheorem{definition}[theorem]{Definition}

\renewcommand{\phi}{\varphi}

\newcommand{\id}{\operatorname{id}}

\newcommand{\Grp}{\mathcal{G}\text{rp}}

\title{Lifting Functors and Relative Schur-Baer Theorems}
\author{Maxim~Ivanov}
\address{Sobolev Institute of Mathematics, 630090 Novosibirsk, Russia}
\email{m.ivanov2@g.nsu.ru}
\thanks{The work was performed according to the Government research assignment for IM SB RAS, project FWNF-2026-0031}

\numberwithin{equation}{section}
\begin{document}
\begin{abstract}
We introduce $\mathcal C$-lifting functors, which axiomatize
lifting properties of non-abelian tensor and exterior products
with respect to prescribed classes of group extensions. For a
homomorphism $f\colon\Gamma\to G$, we associate to every
$\mathcal C$-lifting functor $F$ a relative quotient $F_f(G)$.
This quotient maps epimorphically onto the subgroup determined by
$F$ in every $f$-extension belonging to $\mathcal C$. We show
that the construction is functorial in $f$ and that, for an
$f$-extension $p\colon\widetilde G\to G$, it gives a morphism of
natural exact sequences. In degree two this yields an
epimorphism
$
\frac{H_2(G;\mathbb Z)}{f_*H_2(\Gamma;\mathbb Z)}
\longrightarrow
\ker p\cap[\widetilde G,\widetilde G].
$
Applying the construction to iterated tensor and exterior powers,
we obtain relative Schur--Baer theorems for the lower central and
derived series. We also compare the relative tensor and exterior
squares, relate the exterior construction to the relative Schur
multiplier $H_2(G,\Gamma;\mathbb Z)$, and derive applications to
orderability. As a consequence, we prove that a virtual knot group
is left-orderable if and only if it is circularly orderable.
\end{abstract}
\maketitle
\tableofcontents

\section{Introduction}
\label{sec:introduction}

Schur's theorem states that if $G/Z(G)$ is finite, then the
commutator subgroup $[G,G]$ is finite \cite{Schur1904}. Baer
extended this result to the upper and lower central series: if
$G/Z_n(G)$ is finite, then $\gamma_{n+1}(G)$ is finite
\cite{Baer1952}. These results suggest two natural directions of
generalization. One may replace finiteness by another property of
groups, or replace the lower central series by another commutator
series. Both directions can be studied using non-abelian tensor
and exterior products. The reason is simple: these products come
with natural homomorphisms onto commutator subgroups, and these
homomorphisms lift along suitable extensions.

Let $F\colon\Grp\to\Grp$ be an endofunctor and let
\[
\eta\colon F\Longrightarrow\id
\]
be a natural transformation. Given a class $\mathcal C$ of
homomorphisms, we call $(F,\eta)$ a $\mathcal C$-lifting functor
if for every $p\colon\widetilde G\to G$ in $\mathcal C$ there is
a homomorphism
\[
S\colon F(G)\longrightarrow\widetilde G
\]
such that
\[
pS=\eta_G,
\qquad
SF(p)=\eta_{\widetilde G}.
\]
Iterated tensor and exterior powers are the main examples of such functors
\cite{BrownLoday1984,BrownLoday1987, DonadzeGarciaMartinez2021}.
The relation between Schur-type theorems and closure properties
of the non-abelian tensor product has also been studied in
recent work of Donadze, Ladra and P\'aez-Guill\'an
\cite{DonadzeLadraPaezGuillan2026}. The point of the present paper
is different: we fix a homomorphism $f$ and construct relative
quotients which simultaneously control all $f$-extensions in a
prescribed class.

Let $f\colon\Gamma\to G$ be a homomorphism. An $f$-extension of
$G$ is an extension $p\colon\widetilde G\to G$ for which $f$
admits a lifting to $\widetilde G$. Such extensions were studied
by Farjoun and Segev in connection with relative homology and
universal extensions of group homomorphisms
\cite{FarjounSegev2017}. Put
\[
H_F(G)=\ker\eta_G
\]
and
\[
N_F(f)=
\left\langle F(f)\bigl(H_F(\Gamma)\bigr)\right\rangle^{F(G)}.
\]
We define
\[
F_f(G)=\frac{F(G)}{N_F(f)}.
\]
The basic result is the following.

\begin{theorem*}
Let $(F,\eta)$ be a $\mathcal C$-lifting functor which sends
morphisms in $\mathcal C$ to epimorphisms. If
$p\colon\widetilde G\to G$ is an $f$-extension belonging to
$\mathcal C$, then there is an epimorphism
\[
F_f(G)\longrightarrow\operatorname{Im}\eta_{\widetilde G}.
\]
Moreover, there is an exact sequence
\[
1\longrightarrow
\frac{H_F(G)}{N_F(f)}
\longrightarrow F_f(G)
\longrightarrow\operatorname{Im}\eta_G
\longrightarrow1.
\]
\end{theorem*}

The quotient $F_f(G)$ is functorial in the homomorphism $f$.
More precisely, the construction extends to the arrow category of
groups. For an $f$-extension $p\colon\widetilde G\to G$, this
functoriality gives a morphism between the exact sequence above
and the corresponding exact sequence in $\widetilde G$. In the
case of the exterior square this gives an epimorphism
\[
\frac{H_2(G;\mathbb Z)}{f_*H_2(\Gamma;\mathbb Z)}
\longrightarrow
\ker p\cap[\widetilde G,\widetilde G].
\]
Thus the relative Schur kernel surjects onto the part of the commutator subgroup of $\widetilde G$
which lies in the kernel of the extension.

We apply the general construction to iterated tensor and exterior
powers. If
\[
p\colon\widetilde G\longrightarrow G
\]
is an $f$-extension with
\[
\ker p\subseteq Z_{n-1}(\widetilde G),
\]
then there is an epimorphism
\[
G_f^{\otimes n}\longrightarrow\gamma_n(\widetilde G).
\]
For the corresponding subgroups associated with the derived
series there is an epimorphism
\[
G_f^{\wedge(n)}\longrightarrow\Gamma_n(\widetilde G).
\]
Together with the natural exact sequences for these relative
products, this gives relative Schur--Baer theorems for every class
of groups closed under extensions and homomorphic images.

In degree two,
\[
G\wedge_fG=
\frac{G\wedge G}{f_*H_2(\Gamma;\mathbb Z)}.
\]
The kernel
\[
\frac{H_2(G;\mathbb Z)}{f_*H_2(\Gamma;\mathbb Z)}
\]
is related to the relative Schur multiplier by the exact sequence
\[
0\longrightarrow
\frac{H_2(G;\mathbb Z)}{f_*H_2(\Gamma;\mathbb Z)}
\longrightarrow H_2(G,\Gamma;\mathbb Z)
\longrightarrow
\ker\bigl(H_1(\Gamma;\mathbb Z)\to H_1(G;\mathbb Z)\bigr)
\longrightarrow0.
\]
We also compare the relative tensor and exterior squares using the
natural epimorphism $G\otimes G\to G\wedge G$.

Finally, the relative construction gives orderability results. If
a central $f$-extension of $G$ is left-orderable and
\[
H_2(G;\mathbb Z)/f_*H_2(\Gamma;\mathbb Z)
\]
is torsion-free, then $G\wedge_fG$ is left-orderable. An analogous
result holds for the tensor square. Orderability criteria for
ordinary tensor and exterior squares were obtained in
\cite{Ivanov2024} and applied there to tabulated virtual knot
groups. The relative construction gives a global consequence: every virtual knot group is left-orderable if and
only if it is circularly orderable. The proof uses the ordering extension of a
circularly orderable group and the homomorphism
$\mathbb Z^2\to G$ defined by a meridian and a longitude. The
corresponding Pontryagin product generates $H_2(G;\mathbb Z)$
\cite{Kim2000}, while the Euler class of the ordering extension
restricts trivially to $\mathbb Z^2$.

The paper is organized as follows. In
Section~\ref{sec:lifting-functors} we introduce lifting functors
and relative lifting quotients. In Section~\ref{sec:functoriality}
we construct the relative functor on the arrow category, study
morphisms of $f$-extensions, and introduce universal lifting
quotients. In Section~\ref{sec:tensor-exterior} we apply the theory to
iterated tensor and exterior powers and compare the relative
tensor and exterior squares. In Section~\ref{sec:relative-homology}
we relate the construction to the relative Schur multiplier of
Farjoun and Segev. Section~\ref{sec:applications} contains
algebraic and topological examples. Finally,
Section~\ref{sec:orderability} gives orderability applications.

\section{Relative lifting quotients}
\label{sec:lifting-functors}
Let $\mathcal C$ be a class of group homomorphisms.

\begin{definition}\label{def:C-lifting-functor}
Let
\(
F\colon\Grp\longrightarrow\Grp
\)
be an endofunctor and let
\(
\eta\colon F\Longrightarrow\operatorname{Id}_{\Grp}
\)
be a natural transformation. We say that $(F,\eta)$ is a
\emph{$\mathcal C$-lifting functor} if, for every homomorphism
\(
p\colon\widetilde G\longrightarrow G
\)
belonging to $\mathcal C$, there is a homomorphism
\[
S\colon F(G)\longrightarrow\widetilde G
\]
such that
\[
pS=\eta_G
\qquad\text{and}\qquad
SF(p)=\eta_{\widetilde G}.
\]
Such a homomorphism $S$ is called a \emph{lifting of $\eta$
along $p$}.
\end{definition}
Put
\[
H_F(G)=\ker\eta_G.
\]
\begin{proposition}\label{prop:lifting-kernel-sequence}
Let $p\colon \widetilde G\to G$ belong to $\mathcal C$, and let
$S\colon F(G)\to\widetilde G$ be a lifting. Then
\[
H_F(\widetilde G)
\xrightarrow{\,F(p)\,}
H_F(G)
\xrightarrow{\,S\,}
\ker p
\]
is a complex. It is exact at $H_F(G)$ whenever $F(p)$ is an
epimorphism.
\end{proposition}

\begin{proof}
For $x\in H_F(\widetilde G)$,
\[
SF(p)(x)=\eta_{\widetilde G}(x)=1.
\]
Also, if $y\in H_F(G)$, then
\[
pS(y)=\eta_G(y)=1,
\]
and hence $S(y)\in\ker p$.

Suppose now that $F(p)$ is an epimorphism and that
$y\in H_F(G)$ satisfies $S(y)=1$. Choose
$x\in F(\widetilde G)$ such that $F(p)(x)=y$. Then
\[
\eta_{\widetilde G}(x)
   =SF(p)(x)
   =S(y)
   =1.
\]
Thus $x\in H_F(\widetilde G)$, and the sequence is exact.
\end{proof}

Let $f\colon\Gamma\to G$ be a homomorphism. Define
\[
N_F(f)=
\left\langle F(f)\bigl(H_F(\Gamma)\bigr)
\right\rangle^{F(G)},
\]
where $\langle - \rangle^{F(G)}$ denotes normal closure in $F(G)$.
Naturality of $\eta$ gives
\[
\eta_GF(f)=f\eta_\Gamma,
\]
and therefore
\[
N_F(f)\subseteq H_F(G).
\]
Set
\[
F_f(G)=\frac{F(G)}{N_F(f)}
\qquad\text{and}\qquad
H_F(f)=\frac{H_F(G)}{N_F(f)}.
\]

The homomorphism $\eta_G$ induces a homomorphism
\[
\eta_f\colon F_f(G)\longrightarrow G.
\]

\begin{proposition}\label{prop:relative-exact-sequence}
There is a natural exact sequence
\[
1\longrightarrow H_F(f)
\longrightarrow F_f(G)
\xrightarrow{\,\eta_f\,}
\operatorname{Im}\eta_G
\longrightarrow 1.
\]
\end{proposition}

\begin{proof}
Since $N_F(f)\subseteq\ker\eta_G$, the map $\eta_f$ is defined.
Its image is $\operatorname{Im}\eta_G$, and
\[
\ker\eta_f=\frac{H_F(G)}{N_F(f)}=H_F(f).
\]
\end{proof}

Recall that an $f$-extension of $G$ is an epimorphism
$p\colon\widetilde G\to G$ for which there is a homomorphism
$\widetilde f\colon\Gamma\to\widetilde G$ satisfying
$p\widetilde f=f$.

\begin{lemma}\label{lem:f-extension-factorization}
Let $p\colon\widetilde G\to G$ be an $f$-extension belonging to
$\mathcal C$, and let $S\colon F(G)\to\widetilde G$ be a
lifting. Then
\[
N_F(f)\subseteq\ker S.
\]
Consequently, $S$ induces a homomorphism
\[
\overline S\colon F_f(G)\longrightarrow\widetilde G.
\]
\end{lemma}

\begin{proof}
Let $\widetilde f\colon\Gamma\to\widetilde G$ be a lifting of
$f$. Then
\[
SF(f)
 =SF(p)F(\widetilde f)
 =\eta_{\widetilde G}F(\widetilde f)
 =\widetilde f\eta_\Gamma.
\]
It follows that $SF(f)$ vanishes on $H_F(\Gamma)$. Since
$\ker S$ is normal in $F(G)$, it contains $N_F(f)$.
\end{proof}

\textbf{From now on through the rest of the paper we assume that $F(p)$ is an epimorphism for every
$p\in\mathcal C$.}

\begin{theorem}\label{thm:relative-lifting}
Let $p\colon\widetilde G\to G$ be an $f$-extension belonging to
$\mathcal C$. Then every lifting
$S\colon F(G)\to\widetilde G$ induces an epimorphism
\[
\overline S\colon
F_f(G)\longrightarrow\operatorname{Im}\eta_{\widetilde G}.
\]
\end{theorem}

\begin{proof}
By Lemma~\ref{lem:f-extension-factorization}, the homomorphism
$S$ factors through $F_f(G)$. Since $F(p)$ is an epimorphism,
\[
S(F(G))
 =S\bigl(F(p)(F(\widetilde G))\bigr)
 =\eta_{\widetilde G}(F(\widetilde G)).
\]
Thus $\operatorname{Im}\overline S
=\operatorname{Im}\eta_{\widetilde G}$.
\end{proof}

\begin{corollary}[Generalized Schur--Baer theorem]
\label{cor:general-schur-baer}
Let $\mathcal P$ be a class of groups closed under extensions
and homomorphic images. Let
$p\colon\widetilde G\to G$ be an $f$-extension belonging to
$\mathcal C$. If
\[
H_F(f)\in\mathcal P
\qquad\text{and}\qquad
\operatorname{Im}\eta_G\in\mathcal P,
\]
then
\[
\operatorname{Im}\eta_{\widetilde G}\in\mathcal P.
\]
\end{corollary}

\begin{proof}
Proposition~\ref{prop:relative-exact-sequence} shows that
$F_f(G)\in\mathcal P$. By
Theorem~\ref{thm:relative-lifting},
$\operatorname{Im}\eta_{\widetilde G}$ is a homomorphic image
of $F_f(G)$.
\end{proof}

\begin{corollary}\label{cor:relative-finite}
Let $p\colon\widetilde G\to G$ be an $f$-extension belonging to
$\mathcal C$. If $H_F(f)$ and $\operatorname{Im}\eta_G$ are
finite, then $\operatorname{Im}\eta_{\widetilde G}$ is finite.
Moreover,
\[
\left|\operatorname{Im}\eta_{\widetilde G}\right|
\leq
|H_F(f)|\,|\operatorname{Im}\eta_G|.
\]
\end{corollary}

\begin{proof}
The assertion follows from
Proposition~\ref{prop:relative-exact-sequence} and
Theorem~\ref{thm:relative-lifting}.
\end{proof}

\section{Functoriality and universal lifting quotients}
\label{sec:functoriality}

Let $(F,\eta)$ be as in the preceding section. Put
\[
\overline F(G)=\frac{F(G)}{H_F(G)}
\]
and denote the quotient map by
\[
q_G\colon F(G)\longrightarrow\overline F(G).
\]
Since $H_F$ is a subfunctor of $F$, the assignment
$G\mapsto\overline F(G)$ defines an endofunctor of $\Grp$. The
homomorphism $\eta_G$ induces a monomorphism
\[
\overline\eta_G\colon\overline F(G)\longrightarrow G.
\]
Thus
\[
\overline\eta\colon\overline F\Longrightarrow\id_{\Grp}
\]
is a natural transformation and
\[
\overline F(G)\cong\operatorname{Im}\eta_G.
\]

\subsection{The relative functor on the arrow category}

Let $\operatorname{Arr}(\Grp)$ be the arrow category of groups.
Its objects are homomorphisms
\[
f\colon\Gamma\longrightarrow G,
\]
and a morphism
\[
(a,b)\colon f_1\longrightarrow f_2
\]
is a commutative square
\[
\begin{tikzcd}
\Gamma_1 \arrow[r,"f_1"] \arrow[d,"a"']
    & G_1 \arrow[d,"b"]\\
\Gamma_2 \arrow[r,"f_2"']
    & G_2.
\end{tikzcd}
\]
Composition is componentwise:
\[
(a_2,b_2)(a_1,b_1)=(a_2a_1,b_2b_1).
\]

For classes $\mathcal D$ and $\mathcal C$ of homomorphisms, let
\[
\widehat{\mathcal D,\mathcal C}
\]
denote the class of morphisms $(a,b)$ in
$\operatorname{Arr}(\Grp)$ for which
\[
a\in\mathcal D,
\qquad
b\in\mathcal C.
\]

We use the same notion of lifting in the arrow category. More
explicitly, suppose that $\mathscr G$ is an endofunctor of
$\operatorname{Arr}(\Grp)$ and that for every arrow $f$ there is
a morphism
\[
\theta_f\colon\mathscr G(f)\longrightarrow f
\]
such that
\[
(a,b)\theta_{f_1}
=
\theta_{f_2}\mathscr G(a,b)
\]
for every $(a,b)\colon f_1\to f_2$. For a class $\mathcal A$ of
morphisms in $\operatorname{Arr}(\Grp)$, we say that
$(\mathscr G,\theta)$ is an $\mathcal A$-lifting functor if for every $u\colon f_1\longrightarrow f_2$ in $\mathcal A$ there is a morphism
\[
L\colon\mathscr G(f_2)\longrightarrow f_1
\]
such that
\[
uL=\theta_{f_2},
\qquad
L\mathscr G(u)=\theta_{f_1}.
\]

We shall use the following elementary observation.

\begin{lemma}\label{lem:cube-diagonals}
Consider a commutative cube
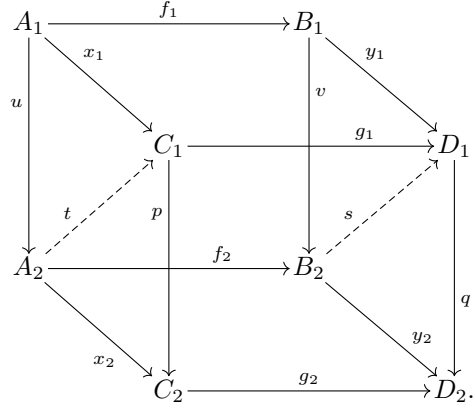
\[
\begin{tikzcd}[
  row sep=3.5em,
  column sep=3.8em,
  cells={nodes={inner sep=1pt}}
]
A_1
  \arrow[rr,"f_1"]
  \arrow[dd,"u"' {pos=0.30}]
  \arrow[dr,"x_1" {pos=0.30}]
&&
B_1
  \arrow[dd,"v" {pos=0.25}]
  \arrow[dr,"y_1" {pos=0.30}]
&
\\
&
C_1
  \arrow[rr,"g_1" {pos=0.72}]
  \arrow[dd,"p"' {pos=0.25}]
&&
D_1
  \arrow[dd,"q" {pos=0.65}]
\\
A_2
  \arrow[rr,swap,"f_2"' {pos=0.72}]
  \arrow[dr,"x_2"' {pos=0.72}]
  \arrow[ur,dashed,"t" {pos=0.30}]
&&
B_2
  \arrow[dr,swap, "y_2"' {pos=0.72}]
  \arrow[ur,dashed,"s" {pos=0.30}]
&
\\
&
C_2
  \arrow[rr,swap,"g_2"']
&&
D_2 .
\end{tikzcd}
\]
Suppose that the dashed diagonals make the left and right faces
commutative. If $u$ is an epimorphism, then
\[
sf_2=g_1t.
\]
\end{lemma}

\begin{proof}
One has
\[
sf_2u
=svf_1
=y_1f_1
=g_1x_1
=g_1tu.
\]
To get the equuality, cancel $u$.
\end{proof}
We want to define a relative functor. For $f\colon\Gamma\to G$, recall the quotient map
\[
\pi_f\colon F(G)\longrightarrow F_f(G).
\]
Since $\pi_fF(f)$ vanishes on $H_F(\Gamma)$, it induces a
homomorphism
\[
\rho_f\colon\overline F(\Gamma)\longrightarrow F_f(G)
\]
determined by $\rho_fq_\Gamma=\pi_fF(f)$.
Define
\[
\mathscr F(f)=
\left(
\overline F(\Gamma)\xrightarrow{\rho_f}F_f(G)
\right).
\]

Let $(a,b)\colon f_1\longrightarrow f_2$
be a morphism of arrows. Naturality of $\eta$ gives
\[
F(a)\bigl(H_F(\Gamma_1)\bigr)
\subseteq H_F(\Gamma_2),
\]
and therefore
\[
F(b)\bigl(N_F(f_1)\bigr)
\subseteq N_F(f_2).
\]
Hence $F(b)$ induces a homomorphism
\[
\beta_{a,b}\colon
F_{f_1}(G_1)\longrightarrow F_{f_2}(G_2)
\]
satisfying
\[
\beta_{a,b}\pi_{f_1}=\pi_{f_2}F(b).
\]
Put
\[
\mathscr F(a,b)=(\overline F(a),\beta_{a,b}).
\]
To see that this is a morphism of arrows, consider a cube
\[
\begin{tikzcd}[
  row sep=3.5em,
  column sep=3.8em,
  cells={nodes={inner sep=1pt}}
]
F(\Gamma_1)
  \arrow[rr,"F(a)"]
  \arrow[dd,"q_{\Gamma_1}"' {pos=0.30}]
  \arrow[dr,"\pi_{f_1}F(f_1)" {pos=0.30}]
&&
F(\Gamma_2)
  \arrow[dd,"q_{\Gamma_2}" {pos=0.25}]
  \arrow[dr,"\pi_{f_2}F(f_2)" {pos=0.30}]
&
\\
&
F_f(G_1)
  \arrow[rr,"\beta_{a,b}" {pos=0.72}]
  \arrow[dd,"="' {pos=0.25}]
&&
F_{f_2}(G_2)
  \arrow[dd,"=" {pos=0.65}]
\\
\overline F(\Gamma_1)
  \arrow[rr,swap,"\overline F(a)"' {pos=0.72}]
  \arrow[dr,"\rho_{f_1}"' {pos=0.72}]
  \arrow[ur,dashed,"\rho_{f_1}" {pos=0.30}]
&&
\overline F(\Gamma_2)
  \arrow[dr,swap, "\rho_{f_2}"' {pos=0.72}]
  \arrow[ur,dashed,"\rho_{f_2}" {pos=0.30}]
&
\\
&
F_{f_1}(G_1)
  \arrow[rr,swap,"\beta_{a,b}"']
&&
F_{f_2}(G_2) .
\end{tikzcd}
\]

Since $q_{\Gamma_1}$ is an epimorphism,
\[
\beta_{a,b}\rho_{f_1}
=
\rho_{f_2}\overline F(a).
\]
The defining equality for $\beta_{a,b}$ also shows immediately
that identities and compositions are preserved. Thus
\[
\mathscr F\colon\operatorname{Arr}(\Grp)
\longrightarrow\operatorname{Arr}(\Grp)
\]
is an endofunctor.

The map $\eta_G$ induces
\[
\eta_f\colon F_f(G)\longrightarrow G,
\qquad
\eta_f\pi_f=\eta_G.
\]
For every $f\colon\Gamma\to G$, the square
\[
\begin{tikzcd}
\overline F(\Gamma)
    \arrow[r,"\rho_f"]
    \arrow[d,"\overline\eta_\Gamma"']
&
F_f(G)
    \arrow[d,"\eta_f"]
\\
\Gamma
    \arrow[r,"f"']
&
G
\end{tikzcd}
\]
commutes. Denote this morphism of arrows by
\[
\widehat\eta_f
=
(\overline\eta_\Gamma,\eta_f)
\colon
\mathscr F(f)\longrightarrow f.
\]
For every $(a,b)\colon f_1\to f_2$ one has
\[
(a,b)\widehat\eta_{f_1}
=
\widehat\eta_{f_2}\mathscr F(a,b),
\]
so these morphisms define a natural transformation
\[
\widehat\eta\colon
\mathscr F\Longrightarrow\id_{\operatorname{Arr}(\Grp)}.
\]

\begin{theorem}\label{thm:relative-arrow-lifting}
Let $\mathcal D$ be a class of homomorphisms such that
$(\overline F,\overline\eta)$ is a $\mathcal D$-lifting functor
and $F(a)$ is an epimorphism for every $a\in\mathcal D$. Then
\[
(\mathscr F,\widehat\eta)
\]
is a $\widehat{\mathcal D,\mathcal C}$-lifting functor.
\end{theorem}

\begin{proof}
Let $(a,b)\colon f_1\longrightarrow f_2$
belong to $\widehat{\mathcal D,\mathcal C}$. Choose liftings
\[
T\colon\overline F(\Gamma_2)\longrightarrow\Gamma_1,
\qquad
S\colon F(G_2)\longrightarrow G_1
\]
along $a$ and $b$.

Write
\[
q_i=q_{\Gamma_i}\colon
F(\Gamma_i)\longrightarrow\overline F(\Gamma_i),
\qquad
\pi_i=\pi_{f_i}\colon
F(G_i)\longrightarrow F_{f_i}(G_i).
\]
Consider the cube
\[
\begin{tikzcd}[
  row sep=3.5em,
  column sep=3.8em,
  cells={nodes={inner sep=1pt}}
]
F(\Gamma_1)
  \arrow[rr,"F(f_1)"]
  \arrow[dd,"F(a)"' {pos=0.30}]
  \arrow[dr,"\eta_{\Gamma_1}" {pos=0.28}]
&&
F(G_1)
  \arrow[dd,"F(b)" {pos=0.24}]
  \arrow[dr,"\eta_{G_1}" {pos=0.28}]
&
\\
&
\Gamma_1
  \arrow[rr,"f_1" {pos=0.72}]
  \arrow[dd,"a"' {pos=0.24}]
&&
G_1
  \arrow[dd,"b" {pos=0.66}]
\\
F(\Gamma_2)
  \arrow[rr,"F(f_2)"' {pos=0.72}]
  \arrow[dr,"\eta_{\Gamma_2}"' {pos=0.72}]
  \arrow[ur,dashed,"Tq_2" {pos=0.30}]
&&
F(G_2)
  \arrow[dr,"\eta_{G_2}"' {pos=0.72}]
  \arrow[ur,dashed,"S" {pos=0.30}]
&
\\
&
\Gamma_2
  \arrow[rr,"f_2"']
&&
G_2 .
\end{tikzcd}
\]
The dashed side faces commute because $T$ and $S$ are liftings.
Since $F(a)$ is an epimorphism,
Lemma~\ref{lem:cube-diagonals} gives
\[
SF(f_2)=f_1Tq_2.
\tag{1}
\]
Hence $S$ vanishes on $F(f_2)(H_F(\Gamma_2))$, and therefore on
$N_F(f_2)$. Thus it induces
\[
\overline S\colon F_{f_2}(G_2)\longrightarrow G_1,
\qquad
\overline S\pi_2=S.
\]

Write
\[
\beta=\beta_{a,b}\colon
F_{f_1}(G_1)\longrightarrow F_{f_2}(G_2).
\]
The lifting identities imply
\[
b\overline S=\eta_{f_2},
\qquad
\overline S\beta=\eta_{f_1}.
\]
Indeed, both equalities follow after precomposition with the
corresponding epimorphism $\pi_i$. We therefore have the cube
\[
\begin{tikzcd}[
  row sep=3.5em,
  column sep=3.8em,
  cells={nodes={inner sep=1pt}}
]
\overline F(\Gamma_1)
  \arrow[rr,"\rho_{f_1}"]
  \arrow[dd,"\overline F(a)"' {pos=0.30}]
  \arrow[dr,"\overline\eta_{\Gamma_1}" {pos=0.28}]
&&
F_{f_1}(G_1)
  \arrow[dd,"\beta" {pos=0.24}]
  \arrow[dr,"\eta_{f_1}" {pos=0.28}]
&
\\
&
\Gamma_1
  \arrow[rr,"f_1" {pos=0.72}]
  \arrow[dd,"a"' {pos=0.24}]
&&
G_1
  \arrow[dd,"b" {pos=0.66}]
\\
\overline F(\Gamma_2)
  \arrow[rr,"\rho_{f_2}"' {pos=0.72}]
  \arrow[dr,"\overline\eta_{\Gamma_2}"' {pos=0.72}]
  \arrow[ur,dashed,"T" {pos=0.30}]
&&
F_{f_2}(G_2)
  \arrow[dr,"\eta_{f_2}"' {pos=0.72}]
  \arrow[ur,dashed,"\overline S" {pos=0.30}]
&
\\
&
\Gamma_2
  \arrow[rr,"f_2"']
&&
G_2 .
\end{tikzcd}
\]
Since $F(a)$ is an epimorphism, so is $\overline F(a)$.
Lemma~\ref{lem:cube-diagonals} gives
\[
\overline S\rho_{f_2}=f_1T.
\]
Thus
\[
(T,\overline S)\colon\mathscr F(f_2)\longrightarrow f_1
\]
is a morphism of arrows. The two dashed side faces give
\[
(a,b)(T,\overline S)=\widehat\eta_{f_2},
\qquad
(T,\overline S)\mathscr F(a,b)=\widehat\eta_{f_1}.
\]
Hence $(T,\overline S)$ is the required lifting.
\end{proof}

\subsection{Morphisms of $f$-extensions}
We first describe the largest class to
which the preceding theorem applies. Put
\[
\mathcal D_F=
\left\{
a\colon E\longrightarrow G
\ \middle|\
F(a)\text{ is an epimorphism and }
\ker a\cap\operatorname{Im}\eta_E=1
\right\}.
\]

\begin{proposition}\label{prop:maximal-D}
The pair $(\overline F,\overline\eta)$ is a
$\mathcal D_F$-lifting functor. Moreover, if $\mathcal D$ is a
class such that $(\overline F,\overline\eta)$ is a
$\mathcal D$-lifting functor and $F(a)$ is an epimorphism for
every $a\in\mathcal D$, then
\[
\mathcal D\subseteq\mathcal D_F.
\]
Consequently,
\[
(\mathscr F,\widehat\eta)
\]
is a $\widehat{\mathcal D_F,\mathcal C}$-lifting functor.
\end{proposition}

\begin{proof}
Let $a\colon E\to G$ belong to $\mathcal D_F$. Since $F(a)$ is
an epimorphism, so is
\[
\overline F(a)\colon\overline F(E)\longrightarrow\overline F(G).
\]
Under the identifications
\[
\overline F(E)\cong\operatorname{Im}\eta_E,
\qquad
\overline F(G)\cong\operatorname{Im}\eta_G,
\]
the map $\overline F(a)$ is the restriction of $a$. The condition
\[
\ker a\cap\operatorname{Im}\eta_E=1
\]
therefore makes $\overline F(a)$ an isomorphism. Define
\[
T=\overline\eta_E\,\overline F(a)^{-1}
\colon\overline F(G)\longrightarrow E.
\]
Then
\[
aT=\overline\eta_G,
\qquad
T\overline F(a)=\overline\eta_E,
\]
so $T$ is a lifting along $a$.

Conversely, let $a\in\mathcal D$ and let
\[
T\colon\overline F(G)\longrightarrow E
\]
be a lifting. Then
\[
T\overline F(a)=\overline\eta_E.
\]
Since $\overline\eta_E$ is injective, $\overline F(a)$ is
injective. It is also an epimorphism because $F(a)$ is one. Hence
$\overline F(a)$ is an isomorphism, and therefore
\[
\ker a\cap\operatorname{Im}\eta_E=1.
\]
Thus $a\in\mathcal D_F$. The last assertion follows from
Theorem~\ref{thm:relative-arrow-lifting}.
\end{proof}

Every isomorphism belongs to $\mathcal D_F$. Hence
$(\mathscr F,\widehat\eta)$ is, in particular, a
\[
\widehat{\operatorname{Iso},\mathcal C}
\text{-lifting functor}.
\]

Let
\[
p\colon\widetilde G\longrightarrow G
\]
be an $f$-extension in $\mathcal C$, defined by a map
\[
\widetilde f\colon\Gamma\longrightarrow\widetilde G.
\]
Then $(\id_\Gamma, p)$ belongs to $\widehat{\operatorname{Iso},\mathcal C}$. Thus the
arrow lifting theorem gives more than the epimorphism of
Theorem~\ref{thm:relative-lifting}: it gives a morphism of the
corresponding extensions.

\begin{corollary}\label{cor:morphism-f-extensions}
Let $p\colon\widetilde G\to G$ be an $f$-extension belonging to
$\mathcal C$, defined by a map
$\widetilde f\colon\Gamma\to\widetilde G$. Then there is a
morphism of $(f\overline\eta_\Gamma)$-extensions
\[
\begin{tikzcd}[
  row sep=3.0em,
  column sep=2.5em,
  cells={nodes={inner sep=1pt}}
]
1 \arrow[r]
&
H_F(f)
  \arrow[r]
  \arrow[dd,two heads]
&
F_f(G)
  \arrow[rr,"\eta_f"]
  \arrow[dd,two heads,"\overline S"']
&&
\operatorname{Im}\eta_G
  \arrow[r]
  \arrow[dd,equal]
&
1
\\
&&&
\overline F(\Gamma)
  \arrow[ul,"\rho_f"]
  \arrow[ur,"f\overline\eta_\Gamma"']
  \arrow[dl,"\widetilde f\,\overline\eta_\Gamma"']
  \arrow[dr,"f\overline\eta_\Gamma"]
&&
\\
1 \arrow[r]
&
\ker p\cap\operatorname{Im}\eta_{\widetilde G}
  \arrow[r]
&
\operatorname{Im}\eta_{\widetilde G}
  \arrow[rr,"p"']
&&
\operatorname{Im}\eta_G
  \arrow[r]
&
1 .
\end{tikzcd}
\]
In particular,
\[
H_F(f)\longrightarrow
\ker p\cap\operatorname{Im}\eta_{\widetilde G}
\]
is an epimorphism.
\end{corollary}

\begin{proof}
Theorem~\ref{thm:relative-arrow-lifting}, applied to the square
\[
\begin{tikzcd}
\Gamma \arrow[r,"\widetilde f"] \arrow[d,equal]
& \widetilde G \arrow[d,"p"]\\
\Gamma \arrow[r,"f"']
& G
\end{tikzcd}
\]
gives
\[
(\overline\eta_\Gamma,\overline S)\colon
\mathscr F(f)\longrightarrow\widetilde f
\]
with
\[
p\overline S=\eta_f,
\qquad
\overline S\rho_f
=\widetilde f\,\overline\eta_\Gamma.
\]
Thus the displayed diagram commutes.

The map $\overline S$ is induced by a lifting
$S\colon F(G)\to\widetilde G$. Since $F(p)$ is an epimorphism
and
\[
SF(p)=\eta_{\widetilde G},
\]
one has
\[
\operatorname{Im}\overline S
=\operatorname{Im}S
=\operatorname{Im}\eta_{\widetilde G}.
\]
Hence the middle vertical map is an epimorphism. The rows are
exact and the right vertical map is the identity, so the induced
map on the kernels is also an epimorphism.
\end{proof}

\subsection{A construction from two lifting functors}

There is a related construction for two lifting functors
connected by a natural transformation. Let
\[
\eta^1\colon F_1\Longrightarrow\id_{\Grp},
\qquad
\eta^2\colon F_2\Longrightarrow\id_{\Grp}
\]
be natural transformations, and let
\[
\alpha\colon F_1\Longrightarrow F_2
\]
satisfy
\[
\eta^2\alpha=\eta^1.
\]
For $f\colon\Gamma\to G$, put
\[
\rho_f^\alpha=F_2(f)\alpha_\Gamma, \qquad
\mathscr F_\alpha(f)=
\left(
F_1(\Gamma)\xrightarrow{\rho_f^\alpha}F_2(G)
\right).
\]
By naturality of $\alpha$,
\[
\rho_f^\alpha=\alpha_GF_1(f).
\]
For $(a,b)\colon f_1\to f_2$, define
\[
\mathscr F_\alpha(a,b)=(F_1(a),F_2(b)).
\]
The square
\[
\begin{tikzcd}
F_1(\Gamma)
    \arrow[r,"\rho_f^\alpha"]
    \arrow[d,"\eta^1_\Gamma"']
&
F_2(G)
    \arrow[d,"\eta^2_G"]
\\
\Gamma \arrow[r,"f"']
& G
\end{tikzcd}
\]
commutes. These squares define a natural transformation
\[
\widehat\eta\colon
\mathscr F_\alpha\Longrightarrow\id_{\operatorname{Arr}(\Grp)}.
\]

\begin{theorem}\label{thm:arrow-lifting-natural-transformation}
Suppose that $(F_1,\eta^1)$ is a $\mathcal D$-lifting functor,
$(F_2,\eta^2)$ is a $\mathcal C$-lifting functor, and $F_1(a)$
is an epimorphism for every $a\in\mathcal D$. Then
\[
(\mathscr F_\alpha,\widehat\eta)
\]
is a $\widehat{\mathcal D,\mathcal C}$-lifting functor.
\end{theorem}

\begin{proof}
Let $(a,b)\colon f_1\to f_2$ belong to
$\widehat{\mathcal D,\mathcal C}$. Choose liftings
\[
T\colon F_1(\Gamma_2)\longrightarrow\Gamma_1,
\qquad
S\colon F_2(G_2)\longrightarrow G_1
\]
along $a$ and $b$. They form the dashed diagonals in the cube
\[
\begin{tikzcd}[
  row sep=3.5em,
  column sep=3.8em,
  cells={nodes={inner sep=1pt}}
]
F_1(\Gamma_1)
  \arrow[rr,"\rho_{f_1}^\alpha"]
  \arrow[dd,"F_1(a)"' {pos=0.30}]
  \arrow[dr,"\eta^1_{\Gamma_1}" {pos=0.28}]
&&
F_2(G_1)
  \arrow[dd,"F_2(b)" {pos=0.24}]
  \arrow[dr,"\eta^2_{G_1}" {pos=0.28}]
&
\\
&
\Gamma_1
  \arrow[rr,"f_1" {pos=0.72}]
  \arrow[dd,"a"' {pos=0.24}]
&&
G_1
  \arrow[dd,"b" {pos=0.66}]
\\
F_1(\Gamma_2)
  \arrow[rr,"\rho_{f_2}^\alpha"' {pos=0.72}]
  \arrow[dr,"\eta^1_{\Gamma_2}"' {pos=0.72}]
  \arrow[ur,dashed,"T" {pos=0.30}]
&&
F_2(G_2)
  \arrow[dr,"\eta^2_{G_2}"' {pos=0.72}]
  \arrow[ur,dashed,"S" {pos=0.30}]
&
\\
&
\Gamma_2
  \arrow[rr,"f_2"']
&&
G_2 .
\end{tikzcd}
\]
Since $F_1(a)$ is an epimorphism,
Lemma~\ref{lem:cube-diagonals} gives
\[
S\rho_{f_2}^\alpha=f_1T.
\]
Thus $(T,S)\colon\mathscr F_\alpha(f_2)\to f_1$ is a morphism of
arrows. The two dashed side faces give
\[
(a,b)(T,S)=\widehat\eta_{f_2},
\qquad
(T,S)\mathscr F_\alpha(a,b)=\widehat\eta_{f_1}.
\]
\end{proof}

\subsection{Universal lifting quotients}

Let $\mathcal B(G)$ be a nonempty class of morphisms
\[
p\colon E\longrightarrow G
\]
belonging to $\mathcal C$. Consider all liftings
\[
S\colon F(G)\longrightarrow E
\]
along morphisms in $\mathcal B(G)$, and put
\[
A_{\mathcal B}(G)=\bigcap_{(p,S)}\ker S.
\]
The intersection is taken over all such pairs $(p,S)$. Define
\[
F_{\mathcal B}(G)=\frac{F(G)}{A_{\mathcal B}(G)}.
\]
Since $pS=\eta_G$, one has
\[
\ker S\subseteq H_F(G)
\]
for every lifting $S$. Hence
\[
A_{\mathcal B}(G)\subseteq H_F(G),
\]
and $\eta_G$ induces a homomorphism
\[
\eta_{\mathcal B,G}\colon F_{\mathcal B}(G)\longrightarrow G.
\]

\begin{proposition}\label{prop:universal-lifting-quotient}
The group $F_{\mathcal B}(G)$ has the following universal
property.
\begin{enumerate}
\item Every lifting associated with a morphism in
$\mathcal B(G)$ factors through $F_{\mathcal B}(G)$.

\item Let
\[
q\colon F(G)\longrightarrow Q
\]
be an epimorphism through which every such lifting factors. Then
there is a unique epimorphism
\[
Q\longrightarrow F_{\mathcal B}(G)
\]
compatible with the quotient maps.
\end{enumerate}
\end{proposition}

\begin{proof}
The first assertion follows from
\[
A_{\mathcal B}(G)\subseteq\ker S
\]
for every lifting $S$. For the second assertion, put $K=\ker q$.
Since every lifting factors through $q$,
\[
K\subseteq\ker S
\]
for every pair $(p,S)$. Hence
\[
K\subseteq A_{\mathcal B}(G),
\]
which gives the required epimorphism. Uniqueness follows because
$q$ is an epimorphism.
\end{proof}

Thus $F_{\mathcal B}(G)$ is the smallest quotient of $F(G)$
through which all liftings associated with $\mathcal B(G)$
factor.

\begin{proposition}\label{prop:universal-quotient-isomorphism}
Suppose that the classes $\mathcal B(G)$ are preserved under
isomorphisms of their codomains. If $G_1\cong G_2$, then
\[
F_{\mathcal B}(G_1)\cong F_{\mathcal B}(G_2).
\]
\end{proposition}

\begin{proof}
Let $\phi\colon G_1\to G_2$ be an isomorphism. If
$p\colon E\to G_1$ belongs to $\mathcal B(G_1)$ and
$S\colon F(G_1)\to E$ is a lifting, then
\[
SF(\phi^{-1})\colon F(G_2)\longrightarrow E
\]
is a lifting along $\phi p$. Moreover,
\[
\ker\bigl(SF(\phi^{-1})\bigr)=F(\phi)(\ker S).
\]
Applying the same argument to $\phi^{-1}$ gives
\[
F(\phi)\bigl(A_{\mathcal B}(G_1)\bigr)
=A_{\mathcal B}(G_2),
\]
and therefore the required isomorphism.
\end{proof}

Let $\mathcal B_f(G)$ be a nonempty class of $f$-extensions of
$G$ belonging to $\mathcal C$.

\begin{proposition}\label{prop:relative-universal-comparison}
There is a canonical epimorphism
\[
F_f(G)\longrightarrow F_{\mathcal B_f}(G).
\]
\end{proposition}

\begin{proof}
By Lemma~\ref{lem:f-extension-factorization}, every lifting
associated with an $f$-extension is trivial on $N_F(f)$. Hence
\[
N_F(f)\subseteq A_{\mathcal B_f}(G),
\]
and the assertion follows.
\end{proof}

Thus the two relative constructions have different roles. The
group $F_f(G)$ is functorial in $f$, whereas
$F_{\mathcal B_f}(G)$ is universal for the chosen class of
$f$-extensions.

\begin{corollary}\label{cor:universal-lifting-image}
Let $p\colon E\to G$ belong to $\mathcal B(G)$. Every lifting
along $p$ induces an epimorphism
\[
F_{\mathcal B}(G)\longrightarrow\operatorname{Im}\eta_E.
\]
\end{corollary}

\begin{proof}
Let $S\colon F(G)\to E$ be a lifting. By
Proposition~\ref{prop:universal-lifting-quotient}, it factors
through $F_{\mathcal B}(G)$. Since $F(p)$ is an epimorphism,
\[
S(F(G))
=S\bigl(F(p)(F(E))\bigr)
=\eta_E(F(E)).
\]
Hence the induced homomorphism has image
$\operatorname{Im}\eta_E$.
\end{proof}

\section{Tensor and exterior powers}
\label{sec:tensor-exterior}

Let
\[
\gamma_1(G)=G,
\qquad
\gamma_{n+1}(G)=[\gamma_n(G),G],
\]
and
\[
\Gamma_1(G)=G,
\qquad
\Gamma_{n+1}(G)=[\Gamma_n(G),\Gamma_n(G)]
\]
be the lower central and derived series of $G$, respectively.

For $n\geq 1$, put
\[
\mathfrak D_n(G)=
\left\{
g\in G\ \middle|\
[\cdots[[g,x_1],x_2],\ldots,x_n]=1
\text{ for all }x_i\in\Gamma_i(G)
\right\}.
\]

For $n\geq 1$, let $G^{\otimes n}$ be the iterated
non-abelian tensor power, where
\[
G^{\otimes 1}=G,
\qquad
G^{\otimes(n+1)}=G^{\otimes n}\otimes G,
\]
with the standard compatible actions. There is a natural
homomorphism
\[
\lambda_n^G\colon G^{\otimes n}\longrightarrow G
\]
whose image is $\gamma_n(G)$.

Let
\[
G^{\wedge(1)}=G,
\qquad
G^{\wedge(n+1)}
=
G^{\wedge(n)}\wedge G^{\wedge(n)}.
\]
There is a natural homomorphism
\[
\mu_n^G\colon G^{\wedge(n)}\longrightarrow G
\]
whose image is $\Gamma_n(G)$. We refer to
\cite{DonadzeGarciaMartinez2021} for the definitions of the
actions and the natural homomorphisms.

For $n\geq2$, let $\mathcal Z_n$ be the class of
epimorphisms
\[
p\colon E\longrightarrow G
\]
such that
\[
\ker p\subseteq Z_{n-1}(E),
\]
and let $\mathcal D_n$ be the class of epimorphisms such that
\[
\ker p\subseteq\mathfrak D_{n-1}(E).
\]

\begin{proposition}\label{prop:tensor-exterior-lifting}
For every $n\geq2$:
\begin{enumerate}
\item
$(G^{\otimes n},\lambda_n)$ is a
$\mathcal Z_n$-lifting functor;

\item
$(G^{\wedge(n)},\mu_n)$ is a
$\mathcal D_n$-lifting functor;

\item
both functors send epimorphisms to epimorphisms.
\end{enumerate}
\end{proposition}

\begin{proof}
Let $p\colon E\to G$ belong to $\mathcal Z_n$. The
homomorphism $\lambda_n^E$ vanishes on the kernel of
\[
p^{\otimes n}\colon E^{\otimes n}\longrightarrow G^{\otimes n}
\]
and hence induces a homomorphism
\[
S_p\colon G^{\otimes n}\longrightarrow E
\]
such that
\[
pS_p=\lambda_n^G,
\qquad
S_pp^{\otimes n}=\lambda_n^E.
\]
The exterior case is proved in the same way. These are
\cite[Lemmas~3.1 and~3.7]{DonadzeGarciaMartinez2021}, with
the indices shifted by one. Preservation of epimorphisms
follows inductively from the exact sequences for
non-abelian tensor and exterior products.
\end{proof}

Let $f\colon\Gamma\to G$ be a homomorphism. Put
\[
K_n^{\otimes}(G)=\ker\lambda_n^G
\]
and
\[
N_n^{\otimes}(f)=
\left\langle
f^{\otimes n}\bigl(K_n^{\otimes}(\Gamma)\bigr)
\right\rangle^{G^{\otimes n}}.
\]
Define
\[
G_f^{\otimes n}
=
\frac{G^{\otimes n}}{N_n^{\otimes}(f)}
\]
and
\[
M_n^{\otimes}(f)
=
\frac{K_n^{\otimes}(G)}{N_n^{\otimes}(f)}.
\]

Similarly, put
\[
K_n^{\wedge}(G)=\ker\mu_n^G,
\]
\[
N_n^{\wedge}(f)=
\left\langle
f^{\wedge(n)}\bigl(K_n^{\wedge}(\Gamma)\bigr)
\right\rangle^{G^{\wedge(n)}},
\]
and define
\[
G_f^{\wedge(n)}
=
\frac{G^{\wedge(n)}}{N_n^{\wedge}(f)},
\qquad
M_n^{\wedge}(f)
=
\frac{K_n^{\wedge}(G)}{N_n^{\wedge}(f)}.
\]

Proposition~\ref{prop:relative-exact-sequence} gives exact
sequences
\[
1\longrightarrow M_n^{\otimes}(f)
\longrightarrow G_f^{\otimes n}
\longrightarrow\gamma_n(G)
\longrightarrow1
\]
and
\[
1\longrightarrow M_n^{\wedge}(f)
\longrightarrow G_f^{\wedge(n)}
\longrightarrow\Gamma_n(G)
\longrightarrow1.
\]

\begin{theorem}\label{thm:relative-tensor-exterior}
Let $p\colon\widetilde G\longrightarrow G$ be an $f$-extension.

\begin{enumerate}
\item
If $\ker p\subseteq Z_{n-1}(\widetilde G)$,
then there is an epimorphism
\[
G_f^{\otimes n}
\longrightarrow
\gamma_n(\widetilde G).
\]

\item
If $\ker p\subseteq\mathfrak D_{n-1}(\widetilde G)$,
then there is an epimorphism
\[
G_f^{\wedge(n)}
\longrightarrow
\Gamma_n(\widetilde G).
\]
\end{enumerate}
\end{theorem}

\begin{proof}
Apply Theorem~\ref{thm:relative-lifting} to the lifting
functors of
Proposition~\ref{prop:tensor-exterior-lifting}.
\end{proof}

\begin{corollary}[Relative Schur--Baer theorem]
\label{cor:relative-schur-baer}
Let $\mathcal P$ be a class of groups closed under extensions
and homomorphic images, and let
$p\colon\widetilde G\to G$ be an $f$-extension.

\begin{enumerate}
\item
If
\[
\ker p\subseteq Z_{n-1}(\widetilde G),
\qquad
M_n^{\otimes}(f)\in\mathcal P,
\qquad
\gamma_n(G)\in\mathcal P,
\]
then
\[
\gamma_n(\widetilde G)\in\mathcal P.
\]

\item
If
\[
\ker p\subseteq\mathfrak D_{n-1}(\widetilde G),
\qquad
M_n^{\wedge}(f)\in\mathcal P,
\qquad
\Gamma_n(G)\in\mathcal P,
\]
then
\[
\Gamma_n(\widetilde G)\in\mathcal P.
\]
\end{enumerate}
\end{corollary}

\begin{proof}
The assertion follows from the two exact sequences above and
Theorem~\ref{thm:relative-tensor-exterior}.
\end{proof}

\begin{corollary}[Relative Baer theorem]
\label{cor:relative-baer}
Let $p\colon\widetilde G\to G$ be an $f$-extension such that
\[
\ker p\subseteq Z_{n-1}(\widetilde G).
\]
If $M_n^{\otimes}(f)$ and $\gamma_n(G)$ are finite, then
$\gamma_n(\widetilde G)$ is finite. Moreover,
\[
|\gamma_n(\widetilde G)|
\leq
|M_n^{\otimes}(f)|\,|\gamma_n(G)|.
\]
\end{corollary}

\begin{corollary}[Relative Baer theorem for the derived series]
\label{cor:relative-baer-derived}
Let $p\colon\widetilde G\to G$ be an $f$-extension such that
\[
\ker p\subseteq\mathfrak D_{n-1}(\widetilde G).
\]
If $M_n^{\wedge}(f)$ and $\Gamma_n(G)$ are finite, then
$\Gamma_n(\widetilde G)$ is finite. Moreover,
\[
|\Gamma_n(\widetilde G)|
\leq
|M_n^{\wedge}(f)|\,|\Gamma_n(G)|.
\]
\end{corollary}

For $n=2$, one has
\[
G^{\wedge(2)}=G\wedge G,
\qquad
\operatorname{Im}\mu_2^G=[G,G],
\qquad
K_2^{\wedge}(G)=H_2(G;\mathbb Z).
\]
Since $H_2(G;\mathbb Z)$ is central in $G\wedge G$,
\[
N_2^{\wedge}(f)=f_*H_2(\Gamma;\mathbb Z).
\]
Thus
\[
G_f^{\wedge(2)}
=G\wedge_fG
=\frac{G\wedge G}{f_*H_2(\Gamma;\mathbb Z)}.
\]

The arrow-category construction gives a sharper form of the
relative Schur theorem: it also describes the part of the
commutator subgroup which lies in the kernel of the extension.

\begin{corollary}[Relative Schur theorem]
\label{cor:relative-schur}
Let
\[
p\colon\widetilde G\longrightarrow G
\]
be a central $f$-extension, defined by a map
\[
\widetilde f\colon\Gamma\longrightarrow\widetilde G.
\]
Then there is a morphism of $f'$-extensions
\[
\begin{tikzcd}[
  row sep=3.0em,
  column sep=2.5em,
  cells={nodes={inner sep=1pt}}
]
1 \arrow[r]
&
\dfrac{H_2(G;\mathbb Z)}{f_*H_2(\Gamma;\mathbb Z)}
  \arrow[r]
  \arrow[dd,two heads]
&
G\wedge_fG
  \arrow[rr]
  \arrow[dd,two heads,"\overline S"']
&&
{[G,G]}
  \arrow[r]
  \arrow[dd,equal]
&
1
\\
&&&
{[\Gamma,\Gamma]}
  \arrow[ul,"\rho_f"]
  \arrow[ur,"f'"]
  \arrow[dl,"\widetilde f'"']
  \arrow[dr,"f'"']
&&
\\
1 \arrow[r]
&
{\ker p\cap[\widetilde G,\widetilde G]}
  \arrow[r]
&
{[\widetilde G,\widetilde G]}
  \arrow[rr,"p"']
&&
{[G,G]}
  \arrow[r]
&
1 .
\end{tikzcd}
\]
Here
\[
f'\colon[\Gamma,\Gamma]\longrightarrow[G,G],
\qquad
\widetilde f'\colon[\Gamma,\Gamma]
\longrightarrow[\widetilde G,\widetilde G]
\]
are induced by $f$ and $\widetilde f$, respectively. In
particular, there are epimorphisms
\[
G\wedge_fG\longrightarrow[\widetilde G,\widetilde G]
\]
and
\[
\frac{H_2(G;\mathbb Z)}{f_*H_2(\Gamma;\mathbb Z)}
\longrightarrow
\ker p\cap[\widetilde G,\widetilde G].
\]
Consequently, if
\[
\frac{H_2(G;\mathbb Z)}{f_*H_2(\Gamma;\mathbb Z)}
\]
and $[G,G]$ are finite, then $[\widetilde G,\widetilde G]$ is
finite. Moreover,
\[
|[\widetilde G,\widetilde G]|
\leq
\left|
\frac{H_2(G;\mathbb Z)}{f_*H_2(\Gamma;\mathbb Z)}
\right|
\,|[G,G]|.
\]
\end{corollary}

\begin{proof}
Apply Corollary~\ref{cor:morphism-f-extensions} to the exterior
square functor with the commutator map
\[
G\wedge G\longrightarrow G.
\]
For this functor
\[
\overline F(\Gamma)\cong[\Gamma,\Gamma],
\qquad
H_F(f)=
\frac{H_2(G;\mathbb Z)}{f_*H_2(\Gamma;\mathbb Z)}.
\]
Since exterior squares send epimorphisms to epimorphisms, the
corollary gives the displayed morphism of extensions. The final
estimate follows from the top exact sequence and the epimorphism
onto $[\widetilde G,\widetilde G]$.
\end{proof}

\subsection{The natural map from the tensor square to the exterior square}

Let
\[
\tau_G\colon G\otimes G\longrightarrow G\wedge G
\]
be the natural epimorphism. The tensor and exterior commutator
maps satisfy
\[
(G\wedge G\longrightarrow G)\tau_G
=
(G\otimes G\longrightarrow G).
\]
Thus $\tau$ is a natural transformation compatible with the two
augmentations.

For $n=2$ one has $\mathcal D_2=\mathcal Z_2$, since
$\mathfrak D_1(G)=Z(G)$. Hence Theorem~
\ref{thm:arrow-lifting-natural-transformation} applies to the
natural transformation $\tau$.

\begin{corollary}\label{cor:tensor-exterior-arrow-lifting}
Define
\[
\mathscr F_\tau(f)=
\left(
\Gamma\otimes\Gamma
\xrightarrow{\rho_f^\tau}
G\wedge G
\right),
\qquad
\rho_f^\tau=(f\wedge f)\tau_\Gamma
=\tau_G(f\otimes f).
\]
Together with the commutator maps this defines a
$\widehat{\mathcal Z_2,\mathcal Z_2}$-lifting functor on
$\operatorname{Arr}(\Grp)$.
\end{corollary}

\begin{proof}
The tensor and exterior square functors, with their commutator
maps, are $\mathcal Z_2$-lifting functors by
Proposition~\ref{prop:tensor-exterior-lifting}, and the tensor
square sends epimorphisms to epimorphisms. Apply
Theorem~\ref{thm:arrow-lifting-natural-transformation}.
\end{proof}

Put
\[
J_2(G)=\ker\bigl(G\otimes G\longrightarrow[G,G]\bigr),
\qquad
\nabla(G)=\ker\tau_G.
\]
The map $\tau_G$ sends $J_2(G)$ onto $H_2(G;\mathbb Z)$.

Since $J_2(G)$ is central in $G\otimes G$, for every
$f\colon\Gamma\to G$ we may write
\[
G\otimes_fG
:=G_f^{\otimes2}
=
\frac{G\otimes G}{(f\otimes f)(J_2(\Gamma))}.
\]
Naturality of $\tau$ gives an induced epimorphism
\[
\tau_f\colon G\otimes_fG\longrightarrow G\wedge_fG.
\]

\begin{proposition}\label{prop:relative-tensor-exterior-map}
Put
\[
\nabla_f(G)=\ker\tau_f.
\]
Then there is an exact sequence
\[
1\longrightarrow\nabla_f(G)
\longrightarrow G\otimes_fG
\xrightarrow{\,\tau_f\,}G\wedge_fG
\longrightarrow1,
\]
and
\[
\nabla_f(G)
\cong
\frac{\nabla(G)}
{\nabla(G)\cap(f\otimes f)(J_2(\Gamma))}.
\]
\end{proposition}

\begin{proof}
The map $\tau_G$ sends
$(f\otimes f)(J_2(\Gamma))$ onto
$f_*H_2(\Gamma;\mathbb Z)$, so it induces the epimorphism
$\tau_f$. Its kernel is the image of $\nabla(G)$ in
$G\otimes_fG$, which is exactly the displayed quotient.
\end{proof}

\begin{corollary}\label{cor:relative-tensor-exterior-epi-f}
If $f\colon\Gamma\to G$ is an epimorphism, then
\[
G\otimes_fG\cong G\wedge_fG.
\]
\end{corollary}

\begin{proof}
The group $\nabla(G)$ is generated by the diagonal elements
$g\otimes g$. If $g=f(x)$, then
\[
g\otimes g=(f\otimes f)(x\otimes x),
\]
and $x\otimes x\in J_2(\Gamma)$. Hence
\[
\nabla(G)\subseteq(f\otimes f)(J_2(\Gamma)),
\]
so $\nabla_f(G)=1$.
\end{proof}

\section{Relative Schur multipliers}
\label{sec:relative-homology}

Let $f\colon\Gamma\to G$ be a homomorphism. Following
Farjoun and Segev \cite{FarjounSegev2017}, let
\[
H_*(G,\Gamma)
\]
denote the homology of the mapping cone of
\[
Bf\colon B\Gamma\longrightarrow BG.
\]
The corresponding exact sequence contains
\[
H_2(\Gamma)
\xrightarrow{\,f_*\,}
H_2(G)
\longrightarrow
H_2(G,\Gamma)
\longrightarrow
H_1(\Gamma)
\xrightarrow{\,f_*\,}
H_1(G).
\]

\begin{proposition}\label{prop:relative-schur-sequence}
There is a natural exact sequence
\[
0\longrightarrow
\frac{H_2(G)}{f_*H_2(\Gamma)}
\longrightarrow
H_2(G,\Gamma)
\longrightarrow
\ker\bigl(H_1(\Gamma)\to H_1(G)\bigr)
\longrightarrow0.
\]
In particular, if
\[
H_1(\Gamma)\longrightarrow H_1(G)
\]
is injective, then
\[
\frac{H_2(G)}{f_*H_2(\Gamma)}
\cong H_2(G,\Gamma).
\]
\end{proposition}

\begin{proof}
This follows from the exact homology sequence of the mapping
cone of $Bf$.
\end{proof}

Thus the group occurring in
Corollary~\ref{cor:relative-schur} is naturally a subgroup of
the relative Schur multiplier $H_2(G,\Gamma)$. It coincides
with $H_2(G,\Gamma)$ when $f$ is injective on abelianizations.

\begin{corollary}\label{cor:relative-homology-kernel}
Suppose that
\[
H_1(\Gamma)\longrightarrow H_1(G)
\]
is injective. If $p\colon\widetilde G\to G$ is a central
$f$-extension, then there is an epimorphism
\[
H_2(G,\Gamma)
\longrightarrow
\ker p\cap[\widetilde G,\widetilde G].
\]
\end{corollary}

\begin{proof}
By Proposition~\ref{prop:relative-schur-sequence},
\[
H_2(G,\Gamma)
\cong
\frac{H_2(G)}{f_*H_2(\Gamma)}.
\]
Now apply Corollary~\ref{cor:relative-schur}.
\end{proof}

\begin{corollary}\label{cor:relative-homology-schur}
Let
\[
p\colon\widetilde G\longrightarrow G
\]
be a central $f$-extension. If $H_2(G,\Gamma)$ and $[G,G]$
are finite, then $[\widetilde G,\widetilde G]$ is finite.
\end{corollary}

\begin{proof}
By Proposition~\ref{prop:relative-schur-sequence},
\[
H_2(G)/f_*H_2(\Gamma)
\]
is finite. The assertion follows from
Corollary~\ref{cor:relative-schur}.
\end{proof}

\section{Examples and applications}
\label{sec:applications}

\subsection{Central extensions of $\mathbb Z^2$}

Let $G=\mathbb Z^2$
and, for $m\geq 1$, define
\[
f_m\colon\mathbb Z^2\longrightarrow\mathbb Z^2,
\qquad
f_m(a,b)=(ma,b).
\]

\begin{proposition}\label{prop:Z2-example}
Let
\[
p\colon\widetilde G\longrightarrow\mathbb Z^2
\]
be a central $f_m$-extension. Then
$[\widetilde G,\widetilde G]$ is cyclic of order dividing $m$.
\end{proposition}

\begin{proof}
There are natural identifications
\[
H_2(\mathbb Z^2;\mathbb Z)
\cong
\bigwedge\nolimits^2\mathbb Z^2
\cong\mathbb Z.
\]
Under these identifications, $(f_m)_*$ is multiplication by
$m$. Hence
\[
\frac{H_2(\mathbb Z^2;\mathbb Z)}
     {(f_m)_*H_2(\mathbb Z^2;\mathbb Z)}
\cong
\mathbb Z/m\mathbb Z.
\]
Since $\mathbb Z^2$ is abelian,
\[
\ker p\cap[\widetilde G,\widetilde G]
=[\widetilde G,\widetilde G].
\]
Corollary~\ref{cor:relative-schur} therefore gives an
epimorphism
\[
\mathbb Z/m\mathbb Z
\longrightarrow
[\widetilde G,\widetilde G].
\]
The assertion follows.
\end{proof}

\subsection{Fundamental groups of fibrations}

We now give a topological application. Let
\[
F\xrightarrow{i}E\xrightarrow{q}B
\]
be a fibration of connected CW-complexes. Recall that the
homomorphism
\[
i_*\colon\pi_1(F)\longrightarrow\pi_1(E)
\]
has a natural crossed-module structure; in particular,
\[
\ker i_*\subseteq Z(\pi_1(F))
\]
\cite[Section~2.6]{BrownHigginsSivera2011}.

\begin{proposition}\label{prop:fibration-application}
Let
\[
F\xrightarrow{i}E\xrightarrow{q}B
\]
be a fibration of connected CW-complexes with
$\pi_1(B)=1$. Let $u\colon X\to F$ be a map from a connected
CW-complex. Suppose that
\[
\operatorname{Im}
\bigl(H_2(X;\mathbb Z)\longrightarrow H_2(E;\mathbb Z)\bigr)
\]
has finite index and that
\[
[\pi_1(E),\pi_1(E)]
\]
is finite. Then
\[
[\pi_1(F),\pi_1(F)]
\]
is finite.
\end{proposition}

\begin{proof}
The homotopy exact sequence of the fibration gives an exact
sequence
\[
\pi_2(B)\longrightarrow
\pi_1(F)\xrightarrow{i_*}\pi_1(E)\longrightarrow1.
\]
Since $\ker i_*$ is central, this is a central extension.

Put
\[
\Gamma=\pi_1(X),
\qquad
G=\pi_1(E),
\qquad
f=(i\circ u)_*.
\]
The homomorphism
\[
u_*\colon\Gamma\longrightarrow\pi_1(F)
\]
is a lifting of $f$. Thus
\[
i_*\colon\pi_1(F)\longrightarrow G
\]
is a central $f$-extension.

The classifying maps of $X$ and $E$ give a commutative
diagram
\[
\begin{tikzcd}
H_2(X;\mathbb Z)
    \arrow[r,"(i\circ u)_*"]
    \arrow[d,two heads]
&
H_2(E;\mathbb Z)
    \arrow[d,two heads]
\\
H_2(\pi_1(X);\mathbb Z)
    \arrow[r,"f_*"']
&
H_2(\pi_1(E);\mathbb Z).
\end{tikzcd}
\]
The vertical homomorphisms are epimorphisms. It follows that
\[
\frac{H_2(\pi_1(E);\mathbb Z)}
     {f_*H_2(\pi_1(X);\mathbb Z)}
\]
is a homomorphic image of
\[
\frac{H_2(E;\mathbb Z)}
     {\operatorname{Im}
       \bigl(H_2(X;\mathbb Z)\to H_2(E;\mathbb Z)\bigr)}.
\]
It is therefore finite. The assertion now follows from
Corollary~\ref{cor:relative-schur}.
\end{proof}

\section{Orderability of relative tensor and exterior squares}
\label{sec:orderability}

We use the following standard facts. Subgroups of left-orderable
and circularly orderable groups have the same property. An
extension of a left-orderable group by a left-orderable group is
left-orderable, while an extension of a circularly orderable group
by a left-orderable group is circularly orderable
\cite{ClayGhaswala2021}.

We first record a consequence of the general relative lifting
theorem.

\begin{proposition}\label{prop:relative-orderability}
Let $(F,\eta)$ be a $\mathcal C$-lifting functor, and suppose
that $F$ sends morphisms in $\mathcal C$ to epimorphisms. Let
\[
p\colon\widetilde G\longrightarrow G
\]
be an $f$-extension belonging to $\mathcal C$. If $H_F(f)$ is
left-orderable, then:
\begin{enumerate}
\item $F_f(G)$ is left-orderable whenever $\widetilde G$ is
      left-orderable;

\item $F_f(G)$ is circularly orderable whenever $\widetilde G$
      is circularly orderable.
\end{enumerate}
\end{proposition}

\begin{proof}
Let
\[
\overline S\colon F_f(G)\longrightarrow
\operatorname{Im}\eta_{\widetilde G}
\]
be the epimorphism given by
Theorem~\ref{thm:relative-lifting}, and put
\[
K=\ker\overline S.
\]
Since
\[
p\overline S=\eta_f,
\]
we have
\[
K\subseteq\ker\eta_f=H_F(f).
\]
Thus $K$ is left-orderable. Moreover,
$\operatorname{Im}\eta_{\widetilde G}$ is a subgroup of
$\widetilde G$. The assertion follows from the exact sequence
\[
1\longrightarrow K\longrightarrow F_f(G)
\longrightarrow\operatorname{Im}\eta_{\widetilde G}
\longrightarrow1.
\]
\end{proof}

We apply this proposition to the tensor and exterior squares.
Recall that
\[
J_2(G)=\ker\bigl(G\otimes G\longrightarrow[G,G]\bigr)
\]
and that both $J_2(G)$ and
\[
H_2(G;\mathbb Z)
=
\ker\bigl(G\wedge G\longrightarrow[G,G]\bigr)
\]
are central in the corresponding tensor and exterior squares. We
also use the relative squares introduced above:
\[
G\otimes_fG
=
\frac{G\otimes G}{(f\otimes f)(J_2(\Gamma))},
\qquad
G\wedge_fG
=
\frac{G\wedge G}{f_*H_2(\Gamma;\mathbb Z)}.
\]
The following statements are relative versions of the
orderability criteria for non-abelian tensor and exterior
squares obtained in \cite{Ivanov2024}.
\begin{theorem}\label{thm:relative-squares-orderability}
Let
\[
p\colon\widetilde G\longrightarrow G
\]
be a central $f$-extension.

\begin{enumerate}
\item Suppose that
\[
\frac{H_2(G;\mathbb Z)}
     {f_*H_2(\Gamma;\mathbb Z)}
\]
is torsion-free. If $\widetilde G$ is left-orderable,
respectively circularly orderable, then $G\wedge_f G$ is
left-orderable, respectively circularly orderable.

\item Suppose that
\[
\frac{J_2(G)}
     {(f\otimes f)(J_2(\Gamma))}
\]
is torsion-free. If $\widetilde G$ is left-orderable,
respectively circularly orderable, then $G\otimes_f G$ is
left-orderable, respectively circularly orderable.
\end{enumerate}
\end{theorem}

\begin{proof}
The two displayed quotients are quotients of the central groups
$H_2(G;\mathbb Z)$ and $J_2(G)$, respectively. Hence they are
abelian; by assumption they are torsion-free, and therefore
left-orderable. The assertion follows from
Proposition~\ref{prop:relative-orderability}, applied to the
exterior and tensor square functors.
\end{proof}

\subsection{Circular orderings}

Let $c$ be a circular ordering of a group $G$, regarded as a
normalized 2-cocycle
\[
c\colon G\times G\longrightarrow\mathbb Z.
\]
It determines a central extension
\[
1\longrightarrow\mathbb Z
\longrightarrow\widetilde G_c
\longrightarrow G
\longrightarrow1,
\]
where
$
\widetilde G_c=\mathbb Z\times G
$ as sets,
with multiplication defined by
\[
(n,g)(m,h)=(n+m+c(g,h),gh).
\]
The group $\widetilde G_c$ is left-orderable
\cite{ClayGhaswala2021}.

\begin{proposition}\label{prop:ordering-extension}
Let $f\colon\Gamma\to G$ be a homomorphism. If
\[
f^*[c]=0\in H^2(\Gamma;\mathbb Z),
\]
then the ordering extension
\[
\widetilde G_c\longrightarrow G
\]
is an $f$-extension.
\end{proposition}

\begin{proof}
The pullback of the ordering extension along $f$ represents
the class
\[
f^*[c]\in H^2(\Gamma;\mathbb Z).
\]
This extension splits when $f^*[c]=0$. A splitting gives a
homomorphism
\[
\widetilde f\colon\Gamma\longrightarrow\widetilde G_c
\]
such that the composite
\[
\Gamma\xrightarrow{\widetilde f}\widetilde G_c
\longrightarrow G
\]
is $f$.
\end{proof}

\begin{corollary}\label{cor:circular-order-relative-squares}
Let $c$ be a circular ordering of $G$, and let
$f\colon\Gamma\to G$ satisfy
\[
f^*[c]=0\in H^2(\Gamma;\mathbb Z).
\]

\begin{enumerate}
\item If
\[
\frac{H_2(G;\mathbb Z)}
     {f_*H_2(\Gamma;\mathbb Z)}
\]
is torsion-free, then $G\wedge_fG$ is left-orderable.

\item If
\[
\frac{J_2(G)}
     {(f\otimes f)(J_2(\Gamma))}
\]
is torsion-free, then $G\otimes_fG$ is left-orderable.
\end{enumerate}
\end{corollary}

\begin{proof}
By Proposition~\ref{prop:ordering-extension}, the left-orderable
group $\widetilde G_c$ is a central $f$-extension of $G$. Apply
Theorem~\ref{thm:relative-squares-orderability}.
\end{proof}

\subsection{Virtual knot groups}
\label{subsec:virtual-knot-groups}

We first record a consequence of the vanishing of bounded
cohomology for amenable groups.

\begin{lemma}\label{lem:euler-class-torus}
Let $c$ be a circular ordering of a group $G$ regarded as a
normalized 2-cocycle. For every homomorphism
\[
f\colon\mathbb Z^2\longrightarrow G
\]
one has
\[
f^*[c]=0.
\]
\end{lemma}

\begin{proof}
We regard $ c $ as a bounded 2-cocycle. After changing coefficients to $\mathbb R$,
its pullback belongs to the image of the comparison map
\[
H_b^2(\mathbb Z^2;\mathbb R)
\longrightarrow
H^2(\mathbb Z^2;\mathbb R).
\]
Since $\mathbb Z^2$ is amenable,
\[
H_b^2(\mathbb Z^2;\mathbb R)=0.
\] Hence the image of $f^*[c]$ in
$H^2(\mathbb Z^2;\mathbb R)$ is zero. The coefficient
homomorphism
\[
H^2(\mathbb Z^2;\mathbb Z)
\longrightarrow
H^2(\mathbb Z^2;\mathbb R)
\]
is injective, and the assertion follows.
\end{proof}

The following criterion is useful independently of knot
groups.

\begin{theorem}\label{thm:torus-orderability-criterion}
Let $G$ be a circularly orderable group such that $G_{\mathrm{ab}}$
is left-orderable. Suppose that there is a homomorphism
\[
f\colon\mathbb Z^2\longrightarrow G
\]
for which
\[
f_*\colon H_2(\mathbb Z^2;\mathbb Z)
\longrightarrow H_2(G;\mathbb Z)
\]
is surjective. Then $G$ is left-orderable.
\end{theorem}

\begin{proof}
Let
\[
1\longrightarrow\mathbb Z
\longrightarrow\widetilde G_c
\longrightarrow G
\longrightarrow1
\]
be the ordering extension of a circular ordering $c$ of $G$.
The group $\widetilde G_c$ is left-orderable. By
Lemma~\ref{lem:euler-class-torus}, the pullback of this
extension along $f$ splits. Thus it is a central $f$-extension.

Since $f_*$ is surjective,
\[
\frac{H_2(G;\mathbb Z)}
     {f_*H_2(\mathbb Z^2;\mathbb Z)}
=0.
\]
Corollary~\ref{cor:circular-order-relative-squares} shows that
\[
G\wedge_f G
=
\frac{G\wedge G}
     {f_*H_2(\mathbb Z^2;\mathbb Z)}
\]
is left-orderable. On the other hand,
\[
G\wedge_fG
=
\frac{G\wedge G}{H_2(G;\mathbb Z)}
\cong [G,G].
\]
Hence $[G,G]$ is left-orderable. The exact sequence
\[
1\longrightarrow [G,G]
\longrightarrow G
\longrightarrow G_{\mathrm{ab}}
\longrightarrow1
\]
now implies that $G$ is left-orderable.
\end{proof}

Let $K$ be a virtual knot, and let $G_K$ be its group. Denote
by $m$ and $l$ a meridian and a longitude of $K$. For the definition of the group and elements $m$ and $l$ see \cite{Kim2000}. The elements
$m$ and $l$ commute, and hence determine a homomorphism
\[
f_K\colon\mathbb Z^2\longrightarrow G_K,
\qquad
f_K(1,0)=m,
\quad
f_K(0,1)=l.
\]
The image of the generator of
$H_2(\mathbb Z^2;\mathbb Z)$ under $(f_K)_*$ is the
Pontryagin product
\[
\langle m,l\rangle\in H_2(G_K;\mathbb Z).
\]
Kim proved that this element generates $H_2(G_K;\mathbb Z)$
\cite[Theorem~16]{Kim2000}.

\begin{theorem}\label{thm:virtual-knot-orderability}
A virtual knot group is left-orderable if and only if it is
circularly orderable.
\end{theorem}

\begin{proof}
Every left-orderable group is circularly orderable. Conversely,
suppose that $G_K$ is circularly orderable. The homomorphism
\[
(f_K)_*\colon
H_2(\mathbb Z^2;\mathbb Z)
\longrightarrow H_2(G_K;\mathbb Z)
\]
is surjective by \cite[Theorem~16]{Kim2000}. Moreover,
\[
(G_K)_{\mathrm{ab}}\cong\mathbb Z.
\]
The assertion follows from
Theorem~\ref{thm:torus-orderability-criterion}.
\end{proof}

\begin{corollary}\label{cor:wirtinger-orderability}
Let $G$ admit a Wirtinger presentation of deficiency $0$ or
$1$. Then $G$ is left-orderable if and only if it is
circularly orderable.
\end{corollary}

\begin{proof}
By \cite[Theorem~3]{Kim2000}, the group $G$ can be realized
as the group of a virtual knot. Apply
Theorem~\ref{thm:virtual-knot-orderability}.
\end{proof}

%\printbibheading[heading=bibintoc]  
\printbibliography
\end{document}